\documentclass[a4paper, 10pt, DIV10, headinclude=false, footinclude=false]{scrartcl}

\usepackage{amsthm, amsmath, amssymb, amsfonts}
\usepackage[utf8]{inputenc}
\usepackage[english]{babel}
\usepackage{url}
\usepackage{mathtools}

\usepackage{tikz}
\usepackage{tikz-3dplot}
\tdplotsetmaincoords{70}{110}
\usepackage{pgfplots}
\usepackage{pgfplotstable}
\usepackage{booktabs}

\numberwithin{figure}{section}
\numberwithin{table}{section}
\numberwithin{equation}{section}

\newenvironment{abstr}[1]{ \vspace{.05in}\footnotesize
	\parindent .2in
	{\upshape\bfseries #1. }\ignorespaces}{\par\vspace{.1in}}
\newenvironment{Abstract}{\begin{abstr}{Abstract}}{\end{abstr}}
\newenvironment{keywords}{\begin{abstr}{Key words}}{\end{abstr}}

\newtheorem{theorem}{Theorem}[section]
\newtheorem{lemma}[theorem]{Lemma}

\newtheorem{proposition}[theorem]{Proposition}

\theoremstyle{definition}

\newcommand{\Veps}{V_{\varepsilon}}
\newcommand{\Aeps}{A_{\varepsilon}}
\def\dx{\,\text{d}x}

\allowdisplaybreaks[4]

\begin{document}
	
	\title{From Domain Decomposition to Homogenization Theory%
	}
	\author{Daniel Peterseim\footnotemark[2] \and Dora Varga\footnotemark[2] \and Barbara Verf\"urth\footnotemark[2]}
	\date{}
	\maketitle
	
	\renewcommand{\thefootnote}{\fnsymbol{footnote}}
	\footnotetext[2]{Institut für Mathematik, Universität Augsburg, Universitätsstr. 14, D-86159 Augsburg}
	\renewcommand{\thefootnote}{\arabic{footnote}}
	
	\begin{Abstract}
		This paper rediscovers a classical homogenization result for a prototypical linear elliptic boundary value problem with periodically oscillating diffusion coefficient. Unlike classical analytical approaches such as asymptotic analysis, oscillating test functions, or two-scale convergence, the result is purely based on the theory of domain decomposition methods and standard finite elements techniques. The arguments naturally generalize to problems far beyond periodicity and scale separation and we provide a brief overview on such applications.
	\end{Abstract}
	
	
\section{Introduction}
\label{sec:1}

Elliptic boundary value problems with oscillatory coefficients play a key role in the mathematical modelling and simulation of complex multiscale problems, for instance transport processes in porous media or the mechanical analysis of composite and multifunctional materials. The characteristic properties of such processes are determined by a complex interplay of effects on multiple non-separable length and time scales. The challenge is that the resolution of all details on all relevant scales may easily lead to a number of degrees of freedom and computational work in a direct numerical simulation which exceed today's computing resources by multiple orders of magnitude. The observation and prediction of physical phenomena from multiscale models, hence, requires insightful methods that effectively represent unresolved scales, i.e., multiscale methods.  

Homogenization is such a multiscale method. It seeks a simplified model that is able to capture the macroscopic responses of the process adequately by a few localized computations on the microscopic scale. Consider, e.g., prototypical second order linear elliptic model problems with highly oscillatory periodic diffusion coefficients that oscillate at frequency $\varepsilon^{-1}$ for some small parameter $0<\varepsilon\ll 1$. Then, the theory of homogenization shows that there exists a constant coefficient such that the corresponding diffusion process represents the macroscopic behaviour correctly. In practice, this yields a two- or multi-scale method that first computes the effective coefficient which is implicitly given through some PDE on the microscopic periodic cell, and then solves the macroscopic effective PDE. This is done for instance in the Multiscale Finite Element Method \cite{EH09MSFEM} or the Heterogeneous Multiscale Method \cite{AEEV12HMM}. In certain cases, the error of such procedures can be quantified in terms of the microscopic length scale $\varepsilon$. The approach and its theoretical foundation can be generalized to certain classes of non-periodic problems. However, the separation of scales, i.e., the separation of the characteristic frequencies of the diffusion coefficient and macroscopic frequencies of interest, seems to be essential for both theory and computation.  

There is a more recent class of numerical homogenization methods that can deal with arbitrarily rough diffusion coefficients beyond scale separation \cite{MP14LOD,HP13oversampl}. While, at first glance, these methods seemed to be only vaguely connected to classical homogenization theory, the recent paper \cite{GP17lodhom} identifies them as a natural generalization of some new characterization of classical homogenization. Another deep connection, which was always believed to exist in the community of domain decomposition methods, is the one between homogenization and domain decomposition. This one was made precise only recently by Kornhuber and Yserentant \cite{KY16LODiterative,KPY16LODiterative,KPY17direct}. By combining their iterative approach to homogenization and the results of \cite{GP17lodhom}, the present paper illuminates the role of domain decomposition in the theory of homogenization and provides homogenization limits without any advanced compactness arguments or two scale limits. In addition, compared with \cite{GP17lodhom}, we are able to drop a technical assumption on some artificial symmetries of the diffusion coefficient with respect to the periodic cell. 

Our new construction of effective coefficients (see Sections~\ref{sec:2}--\ref{sec:3}) is not necessarily any easier than the classical one. For the simple diffusion model problem, this is merely an instance of mathematical curiosity and we do not mean to rewrite homogenization theory. However, the connection between homogenization theory and domain decomposition and, in particular, the method of proof turn out to be very interesting and, moreover, they unroll their striking potential for problems beyond scale separation and periodicity. Using this approach, new theoretical results could be derived and some of them are briefly discussed in Section~\ref{sec:5}.

\section{Model problem and classical homogenization}
For the sake of illustration we restrict ourselves to the simplest possible yet representative and relevant setting. Let $d=2$, $\Omega = [0,1]^2$ and $\varepsilon \Omega := [0, \varepsilon]^2$. Moreover, let $A_1 \in L^{\infty}(\Omega; \mathbb{R}^{2\times 2})$ be a symmetric, uniformly elliptic, $\Omega$-periodic (matrix-valued) coefficient and let $\Aeps(x):=A_1(\frac{x}{\varepsilon})$, $x \in \Omega$. We denote by $V:=H^1_{\#}(\Omega)_{/\mathbb{R}}$ the equivalence class of $\Omega$-periodic functions in $H^1(\Omega)$ factorised by constants, and similarly, by $\Veps:=H^1_{\#}(\varepsilon \Omega)_{/\mathbb{R}}$ for their $\varepsilon$-periodic counterparts.
The model problem under consideration then reads: given $f\in L^2(\Omega)$, find a function $u_\varepsilon \in V$ such that 
\begin{equation}\label{modelpb}
\int_{\Omega} \Aeps(x) \nabla u_{\varepsilon}(x) \cdot \nabla v(x) \dx = \int_{\Omega} f(x) v(x) \dx,
\end{equation}
for all $v \in V$. In order to ensure the well-posedness of the problem, we assume that $\Aeps\in \mathcal{M}_{\alpha\beta}$, where $\mathcal{M}_{\alpha\beta}$ is defined as
\begin{equation*}
\mathcal{M}_{\alpha \beta} := \{A \in L^{\infty}(\Omega) \mid \alpha |\xi|^2 \leq \xi \cdot A(x)\xi \leq \beta |\xi|^2 \text{ for all } \xi \in \mathbb{R}^2 \text{ and a.a. } x \in \Omega\}.
\end{equation*} 

The idea behind classical homogenization is to look for a so-called effective (homogenized) coefficient $A_0 \in \mathcal{M}_{\alpha \beta}$ so that the solution $u_0 \in V$ of the problem 
\begin{equation}\label{hompb}
\int_{\Omega} A_{0} \nabla u_{0} \cdot \nabla v \dx = \int_{\Omega} f v \dx, 
\end{equation}
for all $v \in V$, represents the limit of the sequence $\{u_{\varepsilon}\}_{\varepsilon>0}$ of solutions of the problem \eqref{modelpb}. 
In general, explicit representations of effective coefficients are not known, except for the simple case of the one-dimensional or (locally) periodic setting. However, the so-called energy method of Murat and Tartar ($\cite{Mur78}$) or the two-scale convergence (\cite{All92twosc}) provide us with the following form
\begin{equation}\label{homcoeff}
\biggl(A_0\biggr)_{kj} = \int_{\Omega} \biggl( A_{1}(x)(e_j +  \nabla w_j(x)) \biggr)\cdot\biggl( e_k + \nabla w_k(x) \biggl)\dx,
\end{equation}
where $w_j$ are defined as the unique solutions in $V$ of the so-called cell problems
\begin{equation*}
\int_{\Omega} A_1(x) \left(\nabla w_j (x) - e_j\right) \cdot \nabla v(x) \dx = 0,
\end{equation*}
for all $v \in V$, with the canonical basis $(e_j)_{j=1}^2$ of $\mathbb{R}^2$. The substitution $x \mapsto \frac{x}{\varepsilon}$ yields 
\begin{align}
\nonumber
0 &= \varepsilon^{-2} \int_{\varepsilon \Omega} A_1\Bigl(\frac{x}{\varepsilon}\Bigr) \Bigl(\nabla \underbrace{w_j \Bigl(\frac{x}{\varepsilon}}_{=:\hat{q}_j(x)}\Bigr) - e_j\Bigr) \cdot \nabla \underbrace{v\Bigl(\frac{x}{\varepsilon}\Bigr)}_{=:v_{\varepsilon}(x)} \dx \\&= \int_{\Omega} \Aeps(x) (\nabla \hat{q}_j(x)-e_j) \cdot \nabla v_{\varepsilon}(x) \dx.\label{epseq}
\end{align}	
Since all functions $v_{\varepsilon}$ in $\Veps$ can be written as $v(\frac{x}{\varepsilon})$ for a certain function $v \in V$, equation \eqref{epseq} yields that the function $\hat{q}_j \in \Veps$ solves 
\begin{align}\label{epseqVeps}
\int_{\Omega} \Aeps(x) (\nabla \hat{q}_j(x)-e_j) \cdot \nabla v_{\varepsilon}(x) \dx = 0,
\end{align}	
for all $v_{\varepsilon} \in \Veps$. Moreover, $\hat{q}_j \in \Veps\subset V$ solves the same problem in the space $V$, i.e., 
\begin{align*}
\int_{\Omega} \Aeps(x) (\nabla \hat{q}_j(x)-e_j) \cdot \nabla v(x) \dx = 0,
\end{align*}	
for all $v \in V$, since the solution of an elliptic model problem with periodic data (coefficient, source function) is also periodic, with the same period. 

\section{Novel characterization of the effective coefficient}
\label{sec:2}
In order to define the effective coefficient from the alternative perspective of finite elements, we first introduce the necessary notation on meshes, spaces, and interpolation operators.

We consider structured triangulations of $\Omega=[0, 1]^2$ as depicted in Figure \ref{fig:triangulation}, 
where the triangles $T$ form the triangulation $\mathcal{T}_H$ and the boldface squares $Q$ are part of the square mesh $\mathcal{Q}_H$. Denote the set of nodes by $\mathcal{N}_{\mathcal{T}_H}=\mathcal{N}_{\mathcal{Q}_H}$.
Since we are working with periodic boundary conditions, we will frequently understand  $\mathcal{Q}_H$ and $\mathcal{T}_H$ as periodic partitions (or partitions of the torus or partitions of the whole $\mathbb{R}^2$), i.e., we identify opposite faces of the unit square.
The parameter $H$ denotes the length of the quadrilaterals and is supposed to be  not smaller than the microscopic length scale $\varepsilon$ of the model problem.

\begin{figure}[h]
	\begin{minipage}{0.45\textwidth}
		\includegraphics[scale=0.16]{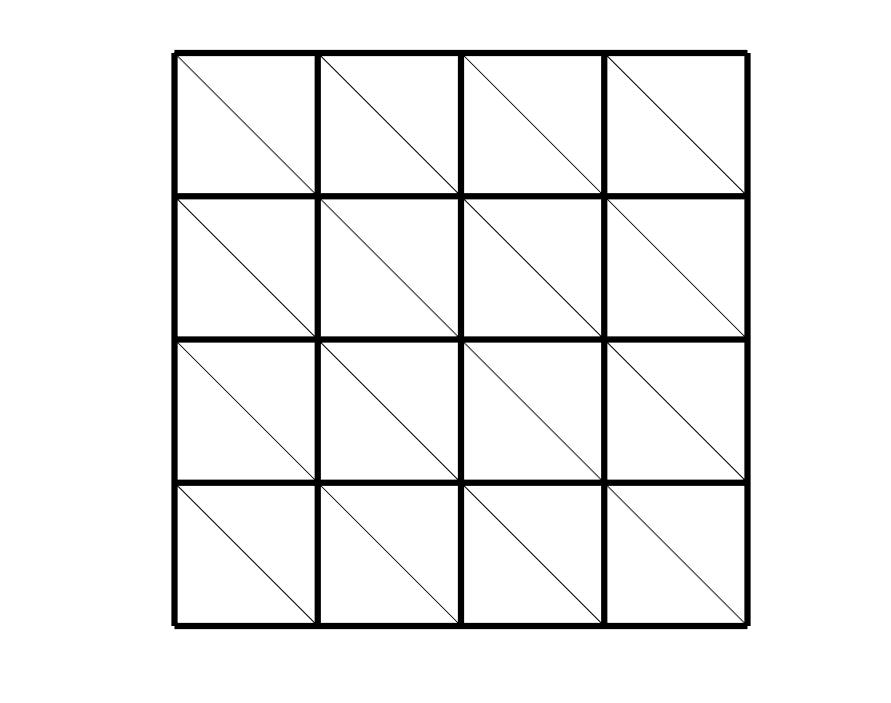}
	\end{minipage}
	\begin{minipage}{0.45\textwidth}
		\includegraphics[scale=0.16]{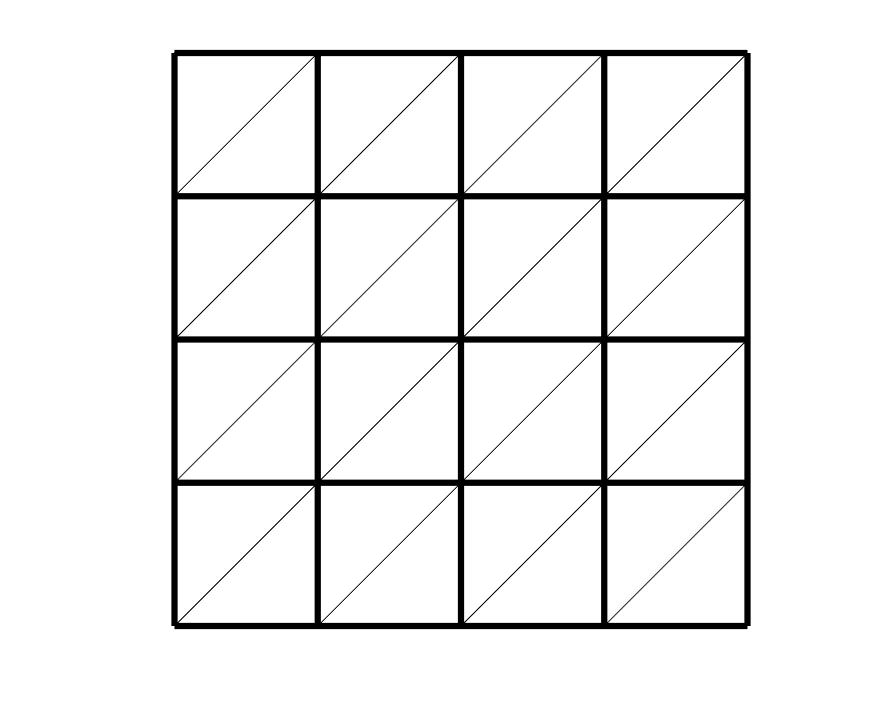}
	\end{minipage}
	\caption{Admissible triangulations.}	\label{fig:triangulation}
\end{figure}

Let  $\mathcal{P}_1(\mathcal{T}_H)$ denote the space of globally continuous piecewise affine functions on $\Omega$ with periodic boundary conditions.
As in the continuous case with $V$, we also factor out the constants here, i.e., in fact we consider $(\mathcal{P}_1(\mathcal{T}_H))_{/\mathbb{R}}$, but still write $\mathcal{P}_1(\mathcal{T}_H)$ for simplicity.
Since $\varepsilon\leq H$ is assumed, the finite element method with the space $\mathcal{P}_1(\mathcal{T}_H)$ does not yield faithful approximations of the solution $u_\varepsilon$ to \eqref{modelpb}; see, e.g., \cite[Sec. 1]{Pet15LODreview}.
We introduce a bounded local linear projection operator $I_H$: $V \rightarrow \mathcal{P}_1(\mathcal{T}_H)$, which can be seen as a composition $I_H:=E_H \circ \Pi_H$, where a function $v \in V$ is first approximated on every element $T \in \mathcal{T}_H$ by its $L^2$-orthogonal projection $\Pi_H$ onto the space of affine functions.
Hence, a possibly globally discontinuous function ${\Pi}_{H}v$ is obtained. In the second step $E_H$, the values at the inner vertices of the triangulation are averages of the respective contributions from the single elements, i.e.,
\begin{equation*}
E_H \circ \Pi_H (v) (z):= \frac{1}{\# \{T \in \mathcal{T}_H, z \in T\}}\sum_{\substack{T \in \mathcal{T}_H \\ z \in T}} \Pi_H (v)|_T (z)
\end{equation*}
for all vertices $z$, where the triangulation is understood in a periodic manner, see \cite{OV16lodtwosc}.

Let $W:=\text{kern} I_H$ be the kernel of the quasi-interpolation operator $I_H$.
It can be seen as the set of rapidly oscillating functions, which cannot be captured by standard finite elements functions on the (coarse) mesh $\mathcal{T}_H$.
Motivated by the reformulation \eqref{epseqVeps} of the cell problems and the interpretation of $W$ as rapidly oscillating functions, 
we define now the correctors $q^\infty_{Q,j}$ as the unique solutions in $W$  of the following variational problems
\begin{align}\label{correceq}
\int_{\Omega} \Aeps(x) \nabla q^{\infty}_{Q,j}(x) \cdot \nabla w(x) \dx = \int_{Q} \Aeps(x) e_j \cdot \nabla w(x) \dx ,
\end{align}	
for all $w \in W$, and the correctors are defined for every $Q \in \mathcal{Q}_H$, $j=1,2$. We define the following w.r.t. $\mathcal{Q}_H$ piecewise constant numerical coefficient $A^{\infty}_H$ which will play the main role in Proposition $\ref{Prop}$:
\begin{equation}
\label{eq:AHinfty}
\biggl[A^{\infty}_{H|Q}\biggr]_{kj} = \frac{1}{|Q|}\int_{Q} \Aeps(x)e_j \cdot e_k \dx - \frac{1}{|Q|} \int_{\Omega} \Aeps(x)  \nabla q^{\infty}_{Q,j}(x) \cdot e_k \dx,
\end{equation}
for all $Q \in \mathcal{Q}_H$, $k,j=1,2$.
\begin{proposition}\label{Prop}
	In the case that the mesh size $H$ is an integer multiple of $\varepsilon$, the coefficient $A^{\infty}_H$ coincides with the homogenized coefficient $A_0$ from classical homogenization defined in \eqref{homcoeff}.
\end{proposition}
\begin{proof}
	We will first show that the function $q_j := \sum_{Q \in \mathcal{Q}_H} q^{\infty}_{Q,j}$ coincides with the corrector $\hat{q}_j \in \Veps$, the unique solution of the problem \eqref{epseqVeps}. 	
	The crucial observation needed for the proof is the fact that the space of $\varepsilon$-periodic functions is contained in the kernel $W$ of the quasi-interpolation operator $I_H$, in the case of the present setting with the triangulations $\mathcal{T}_H$ and $\mathcal{Q}_H$. To see this we observe that, for an $\varepsilon$-periodic function $v_{\varepsilon}\in \Veps$, the values $I_H(v_{\varepsilon})(z)$ coincide
	for all $z\in \mathcal{N}_{\mathcal{T}_H}$. That is, $I_H(v_{\varepsilon})\in \mathcal{P}_1(\mathcal{T}_H)$ is a global constant. As we factored out the constants, we can take the zero function as representative, i.e., $I_H(v_{\varepsilon})=0$.
	
	Moreover, summing up the equations \eqref{correceq} over all $Q \in \mathcal{Q}_H$ and taking advantage of the symmetry of $\Aeps$, we get that $q_j := \sum_{Q \in \mathcal{Q}_H}q_{Q, j}^\infty$ solves
	\begin{align}
	\int_{\Omega} \Aeps(x)\nabla q_j(x) \cdot \nabla w(x) \dx &= \int_{\Omega} \Aeps(x) e_j \cdot \nabla w(x) \dx \\ &= \int_{\Omega} \Aeps(x) \nabla w(x) \cdot e_j \dx, \label{neweq} 
	\end{align} 
	for all $w \in W$, and in particular for all $w\in \Veps$. The combination of \eqref{epseqVeps} and \eqref{neweq} readily yields that $q_j\equiv \hat{q}_j$, $j=1,2$. Moreover, \eqref{neweq} with $w = q^{\infty}_{Q,k}$ implies 
	\begin{align*}
	\int_{\Omega} \Aeps(x)\nabla q^{\infty}_{Q,k}(x) \cdot \nabla e_j \dx &= \int_{\Omega} \Aeps(x) \nabla q_j(x) \cdot \nabla q^{\infty}_{Q,k}(x) \dx \\ &= \int_{\Omega} \Aeps(x) \nabla q^{\infty}_{Q,k}(x) \cdot \nabla q_j(x) \dx \\ &= \int_Q \Aeps(x) e_k \cdot \nabla q_j(x) \dx.  
	\end{align*}
	Hence, in the definition of $A^{\infty}_H$ we can replace the second term, namely
	\begin{align*}
	\biggl(A^{\infty}_{H|Q}\biggr)_{kj} &= \frac{1}{|Q|}\int_{Q} \Aeps(x)e_j \cdot e_k \dx - \frac{1}{|Q|} \int_{Q} \Aeps(x)  e_j \cdot \nabla q_k(x) \dx \\ &= \frac{1}{|Q|} \int_{Q} \Aeps(x)  e_j \cdot (e_k - \nabla \hat{q}_k(x)) \dx \\ &= \biggl(A_{0}\biggr)_{kj},
	\end{align*}
	for $j,k = 1,2$.
\end{proof}

\section{Numerical effective coefficient by domain decomposition}
\label{sec:3}
The correctors $q_{Q, j}^\infty$ defined in the previous section require the solution of a global problem involving the oscillating coefficient $\Aeps$.
Employing domain decomposition, we introduce localized variants and then use arguments from the theory of iterative (domain decomposition) methods as presented in \cite{KPY16LODiterative, KY16LODiterative} to show that the error decays exponentially in the number of iterations.
With the localized correctors, we then introduce an effective localized coefficient $A^\ell_H$ which is piecewise constant on $\mathcal{Q}_H$.

Let $\omega_i$ be the union of all squares $Q\in \mathcal{Q}_H$ having the vertex $z_i$ as a corner and let 
\begin{equation}    \label{Wi} 
W_i=\{v-I_H v\,|\,v\in H^1_0(\omega_i)\}.
\end{equation}
We emphasize that $\omega_i$ is understood as a subset of $\mathbb{R}^2$, i.e., it is continued over the periodic boundary.
The functions in $W_i$ vanish outside a small 
neighbourhood of the vertex $z_i$.
The $W_i$ are closed subspaces of the kernel $W$ 
of $I_H$, see \cite{KPY16LODiterative}.
Let $P_i$ be the $a_\varepsilon$-orthogonal projection from 
$V$ to $W_i$, defined via the 
equation
\begin{equation}    \label{Pi}
a_\varepsilon(P_iv,w_i)=a_\varepsilon(v,w_i),\quad \forall w_i\in W_i.
\end{equation}
Introducing the with respect to the bilinear form 
$a_\varepsilon(\cdot, \cdot)$ symmetric operator
\begin{equation}    \label{P}
P=P_1+P_2+\cdots+P_n,
\end{equation}
the following properties are proved in \cite{KPY16LODiterative}:

\begin{lemma} \label{lemmaP}
	There are constants $K_1$ and $K_2$, independent of $H$ and $\varepsilon$, such that
	\[K_1^{-1}a_\varepsilon(v,v)\leq a_\varepsilon(Pv, v)\leq K_2 a_\varepsilon (v,v)\]
	for all $v\in V$.
	Moreover, for an appropriate scaling factor $\vartheta$ only depending on $K_1$ and $K_2$, there exists a positive constant $\gamma<1$ such that
	\begin{equation} \label{approxP}
	\|\operatorname{id}-\vartheta P\|_{\mathcal{L}(V,V)}\leq \gamma.
	\end{equation}
\end{lemma}

Starting from $q_{Q,j}^0=0$, $j=1,2$, the localized correctors $q_{Q, j}^\ell$ are defined for all $Q \in \mathcal{Q}_H$ via
\begin{equation}  \label{correclocal}
q_{Q, j}^{\ell+1}=q_{Q, j}^\ell+\vartheta P(x_j\operatorname{1}_Q-q_{Q, j}^\ell), \qquad j=1,2,
\end{equation}
where $\operatorname{1}_Q$ denotes the characteristic function of $Q$ and $x_j$ denotes the $j$-th component of the (vector-valued) function $x\mapsto x$.
The scaling factor $\vartheta$ is chosen as discussed in Lemma \ref{lemmaP}.
The correction $P(x_j\operatorname{1}_Q-q_{Q, j}^\ell)$ is the sum of its components $C_i^\ell=P_i(x_j\operatorname{1}_Q-q_{Q, j}^\ell)$ in the subspaces $W_i$ of $W$, where the $C_i^\ell$ solve the local equations
\begin{equation}
\label{correclocaliter}
a_\varepsilon(C_i^\ell, w_i)=a_\varepsilon(x_j\operatorname{1}_Q, w_i)-a_\varepsilon(q_{Q, j}^\ell, w_i), \qquad \forall w_i\in W_i.
\end{equation}
The sloppy notation using $1_Q$ as argument in $a_\varepsilon$ is to denote that the integration is over the element $Q$ only, i.e., $a_\varepsilon(x_j\operatorname{1}_Q, w_i)=\int_Q \Aeps e_j\cdot \nabla w_i\, dx$.
Since the local projections $P_i$ only slightly increase the support of a function, we deduce inductively that the support of $q_{Q, j}^\ell$ is contained in an $\ell H$-neighbourhood of $Q$.
In particular, in each step of \eqref{correclocal} only a few local problems of type \eqref{correclocaliter} have to be solved.

We now replace $q_{Q, j}^\infty$ by its localized variant $q_{Q, j}^\ell$ in the definition of the numerical effective coefficient.
This procedure is justified by an exponential error estimate in Proposition \ref{propAerror}.
We define the piecewise constant (on the mesh $\mathcal{Q}_H$) (localized) effective matrix $A^\ell_H$ via
\begin{equation}
\label{AHl}
\Bigl(A_H^\ell|_{Q}\Bigr)_{k j} = \frac{1}{|Q|}\int_Q \Aeps(x) e_j\cdot e_k\, dx-\frac{1}{|Q|}\int_{\Omega}\Aeps \nabla q_{Q, j}^\ell(x)\cdot e_k\, dx.
\end{equation}

Since the numerical effective coefficient \eqref{eq:AHinfty} is the ``true'' one in the sense that $A_H^\infty=A_0$, we simply need to estimate the error of the iterative approximation.

\begin{proposition}
	\label{propAerror}
	Let H be an integer multiple of $\varepsilon$ and let the localization parameter $\ell$ be chosen of order $\ell\approx |\log H|$. Then,
	\begin{equation}
	\|A_H^\infty-A_H^\ell\|_{L^\infty (\Omega)}\lesssim H.
	\end{equation}
\end{proposition}

\begin{proof}
	We first estimate the error between the correctors $q_{Q, j}^\infty$ and $q_{Q, j}^\ell$.
	Using the definition of $q_{Q, j}^\infty$ in \eqref{correceq}, we deduce that $P(x_j\operatorname{1}_Q)=P(q_{Q, j}^\infty)$.
	Hence, we can characterize the error between the correctors $q_{Q, j}^\infty$ and their localized approximations $q_{Q, j}^\ell$ via
	\[q_{Q, j}^\infty-q_{Q, j}^\ell=(\operatorname{id}-\vartheta P)^\ell q_{Q, j}^\infty.\]
	Using \eqref{approxP}, this yields the exponential convergence of $q_{ Q, j}^\ell$ towards $q_{Q, j}^\infty$, i.e.,
	\begin{equation}\label{approxq}
	\|\nabla(q_{Q, j}^\infty-q_{Q, j}^\ell)\|\lesssim \gamma^\ell \|\nabla q_{Q, j}^\infty\|\lesssim \gamma^\ell |Q|^{1/2}.
	\end{equation}
	
	By the definitions of $A_H^\infty$ in \eqref{eq:AHinfty} and $A_H^\ell$ in \eqref{AHl}, we obtain
	\begin{align*}
	\Bigl|(A_H^\infty|_Q)_{jk}-(A_H^\ell|_Q)_{jk}\Bigr|&=|Q|^{-1}\Biggl|\int_{\Omega}\Aeps\nabla (q_{Q, j}^\ell-q_{Q, j}^\infty)\cdot e_k\, dx\Biggr|\\
	&\lesssim |Q|^{-1} \|e_k\|_{L^2(\Omega)}\|\nabla (q_{Q, j}^\ell-q_{Q, j}^\infty)\|_{L^2(\Omega)}.
	\end{align*}
	Estimate \eqref{approxq} and the choice $\ell\approx |\log H|$ readily imply the assertion.
\end{proof}

The same estimate was previously derived in  \cite{GP17lodhom} with a slightly different localization strategy and with more restrictive conditions on the triangulation. There, the homogenization error in the $L^2$-norm is quantified as follows. 
Let $\Omega$ be convex. Let $u_\varepsilon\in V$ solve \eqref{modelpb} and let $u_0\in V$ be the solution to \eqref{hompb}. For sufficiently small $\varepsilon$, it holds that
\[\|u_\varepsilon-u_0\|_{L^2(\Omega)}\lesssim \varepsilon|\log \varepsilon|^2\|f\|_{L^2(\Omega)}.\]
This estimate recovers the classical result that $u_\varepsilon\to u_0$ strongly in $L^2$ and furthermore states that the convergence is almost linear for right-hand sides $f\in L^2(\Omega)$. We shall emphasize that the proofs of \cite{GP17lodhom} are solely based on standard techniques of finite elements. The authors believe that such a result is also possible in the slightly more general setup of this paper. However, it seems that there is no simple argument but the generalization requires to revise the analysis of \cite{GP17lodhom} step by step which is far beyond the scope of this paper. 

\section{Beyond periodicity and scale separation}
\label{sec:5}

The numerical approach presented in Section \ref{sec:3} does not essentially rely on the assumption of periodicity or separation of scales (between the length scales of the computational domain and the material structures).
Of course, in such general situations, one cannot identify a constant effective coefficient.
Instead the goal is to faithfully approximate the analytical solution by a (generalized) finite element method based on a (coarse) mesh, which does not need to resolve the fine material structures and thereby is computationally efficient.

For this generalization, note that the definition \eqref{correclocal} can be formulated verbatim for any boundary value problem involving a potentially rough, but not necessarily periodic diffusion tensor $A\in L^\infty(\Omega)$.
Moreover, the choice of the function $x_j \operatorname{1}_Q$ in the definition of $q_{Q, j}^\ell$ can be generalized to any function $v\in V$ in the following way. Define the operator $C_T^\ell:V\to W$ inductively via $C_T^0=0$ and
\[C_T^{\ell+1}=C_T^\ell+\vartheta P(\operatorname{id}|_T-C_T^\ell)\]
for all $T\in \mathcal{T}_H$, see \cite{KPY16LODiterative}.
Instead of modifying the diffusion tensor as in the previous sections, we then modify the basis functions and define a generalized finite element method using the test and ansatz spaces $V_H^\ell:=(\operatorname{id}+ C^\ell)\mathcal{P}_1(\mathcal{T}_H)$ with $C^\ell:=\sum_{T\in \mathcal{T}_H}C_T^\ell$.
This method is known as the Localized Orthogonal Decomposition (LOD) \cite{HP13oversampl, MP14LOD,HM14LODbdry,Pet15LODreview} and originally arose from the concept of the Variational Multiscale Method \cite{HFMQ98VMM,HS07VMM}.
Note that mostly a slightly different definition of the correctors $C_T^\ell$ based on patches of diameter $\ell H$ around the element $T$ is used.
The present approach via domain decomposition and iterative solvers was developed recently in \cite{KY16LODiterative,KPY16LODiterative}.
It has been shown in \cite{MP14LOD, HM14LODbdry} for instance, that the method approximates the analytical solution with an energy error of the order $H$ even in the pre-asymptotic regime if the localization parameter $\ell$ is chosen of the order $\ell\approx |\log H|$ as in Proposition \ref{propAerror}.
Hence, the Localized Orthogonal Decomposition can efficiently treat general multiscale problems.
Besides the above mentioned Galerkin-type ansatz with modified ansatz and test functions, Petrov-Galerkin formulations of the method \cite{EGH15LODpetrovgalerkin} may have computational advantages \cite{EHMP16LODimpl} and even meshless methods are possible \cite{HMP14partunity}.

The Localized Orthogonal Decomposition is not restricted to elliptic diffusion problems and has underlined its potential in various applications and with respect to different (computational) challenges.
Starting from the already mentioned application in the geosciences, we underline that the material coefficients are often characterized not only by rapid oscillations but also by a high contrast, i.e., the ratio $\beta/\alpha$ is large.
Many error estimates, also for the standard LOD, are contrast-dependent, but a careful choice of the interpolation operator, see \cite{HM17lodcontrast,PS16lodcontrast}, can overcome this effect.
Apart from simple diffusion problems, porous media \cite{BP16lodporous}, elasticity problems \cite{HP16lodelasticity} or coupling of those such as in poroelasticity \cite{ACMPP18poro} play important roles in these (and many other) applications.
For instance in elasticity theory, not only heterogeneous materials are treated, but also the effect of locking can be reduced by the multiscale method in \cite{HP16lodelasticity}.

Another important area of research are acoustic and electromagnetic wave propagation problems, where the considered prototypical equations are the Helmholtz and Maxwell's equations.
It is well known that standard finite element discretizations of the (indefinite) Helmholtz equation are only well-posed and converging under a rather restrictive resolution condition between the mesh size and the wavenumber.
In a series of paper \cite{BGP15hethelmholtzLOD,GP15scatteringPG,P17LODhelmholtz}, it was analysed that the LOD can relax this resolution condition if the localization parameter grows logarithmically with the wavenumber. For large wavenumbers, this is a great computational gain in comparison to standard numerical methods.
Maxwell's equations, studied in \cite{GHV17lodcurl,Ver17lodcurlindef}, on the other hand, pose a challenge as the involved curl-operator has a large kernel. Moreover, the natural finite element space are N\'ed\'elec's edge elements, for which stable interpolation operators are much less developed than for Lagrange finite elements.
In the context of problems not based on standard Lagrange spaces, we also mention the mixed problem utilizing Raviart-Thomas spaces in \cite{HHM16LODmixed}.
Considering wave problems, the time-dependent wave equation with different time discretizations was studied in \cite{AH17LODwaves, MP18lodwave}. Concerning time-dependency, an important question for the LOD construction is how to deal with time-dependent diffusion tensors.
\cite{HM17lodlagging} presents an a posteriori error estimator in order to adaptively decide which correction to recompute in the next time step.

Apart from the treatment of multiscale coefficients in a variety of partial differential equations, the methodology can also be seen as a stabilization scheme similar as its origin the variational multiscale methods. This has been exploited to deal with the pollution effect in Helmholtz problems mentioned above, for convection dominated diffusion problems \cite{LPS17lodadvec} and, more importantly, to bypass CFL conditions in the context of explicit wave propagation on adaptive meshes \cite{PS17lodwavecfl}. 

Further unexpected applications are linear and nonlinear eigenvalue problems \cite{MP15LODeig,MP17lodeigenvalue}, in particular the quantum-physical simulation based on the Gross-Pitaevskii equation. While the LOD can be employed to speed-up ground state computations for rather rough potentials \cite{HM14LODBOC}, the underlying technique of localization by domain decomposition turned out to be of great value to provide (analytical) insight into the phenomenon of Anderson localization in this context. The recent paper \cite{APP18anderson} predicts and quantifies the emergence of localized eigenstates and might inspire progress regarding the understanding of localization effects which are observed for many other problems as well.

The present contribution aimed at unifying the view of the LOD and classical homogenization and domain decomposition. As already mentioned, close connections exist with \cite{GP17lodhom} and its extension to stochastic homogenization \cite{GP17lodhomstoch}. Further applications involve a multilevel generalization of LOD named gamblets \cite{Owhadi2017} (due to a possible game-theoretic interpretation). This multilevel variant allows surprising results such as a sparse representation of the expected solution operator for random elliptic boundary value problems \cite{FP18expecsol} which may inspire new computational strategies for uncertainty quantification in the future.

\end{document}